\documentclass[11pt,oneside]{amsart}

\title{The set splittability problem}
\author{Peter Bernstein}
\address{Peter Bernstein, Tufts University, Medford, MA 02155}
\email{peter.bernstein@tufts.edu}
\author{Cashous Bortner}
\address{Cashous Bortner, University of Nebraska, Lincoln, NE 68588}
\email{cashous.bortner@huskers.unl.edu}
\author{Samuel Coskey}
\address{Samuel Coskey, Boise State University, Boise, ID 83702}
\email{scoskey@gmail.com}
\author{Shuni Li}
\address{Shuni Li, Macalester College, Saint Paul, MN 55105}
\email{sli@macalester.edu}
\author{Connor Simpson}
\address{Connor Simpson, Cornell University, Ithaca, NY 14853}
\email{cgs93@cornell.edu}

\usepackage[T1]{fontenc}
\usepackage{lmodern}
\usepackage[marginratio=1:1]{geometry}
\usepackage{setspace}
\usepackage{mathpazo,bm,amssymb}

\usepackage{xfrac}
\usepackage{color}
\usepackage{mathtools}
\usepackage{caption}
\usepackage{textcomp}
\usepackage{tikz}
\usetikzlibrary{positioning,shapes.geometric}
\usetikzlibrary{intersections}
\usetikzlibrary{calc}
\usepackage{array}
\usepackage{hyperref}

\newtheorem{thm}{Theorem}[section]
\newtheorem{lem}[thm]{Lemma}
\newtheorem{prop}[thm]{Proposition}

\newtheorem{statement}[thm]{Statement}
\theoremstyle{definition}

\theoremstyle{remark}

\newtheorem{claim}[thm]{Claim}
\newenvironment{claimproof}{%
  \begin{proof}[Proof of claim]%
  }{%
  \end{proof}%
}

\renewcommand{\labelenumi}{(\alph{enumi})}
\renewcommand{\theenumi}{\alph{enumi}}

\usepackage{etoolbox}
\makeatletter\pretocmd{\@seccntformat}{\S}{}{}
  \pretocmd{\@subseccntformat}{\S}{}{}\makeatother
\newcommand{\floor}[1]{\ensuremath{\left\lfloor #1 \right\rfloor}}
\newcommand{\ceiling}[1]{\ensuremath{\left\lceil #1 \right\rceil}}
\newcommand{\nint}[1]{\ensuremath{\left\lfloor #1 \right\rceil}}

\DeclareMathOperator{\N}{\mathbb{N}}
\DeclareMathOperator{\R}{\mathbb{R}}

\renewcommand{\subset}{\subseteq}

\DeclareMathOperator{\disc}{disc}
\newcommand{\B}{\mathcal{B}}
\newcommand{\half}{\frac{1}{2}}

\newcommand{\psplit}{\textsc{$p$-Split}}
\newcommand{\zoe}{\textsc{ZOE}}
\renewcommand{\vec}[1]{\mathbf{#1}}

\newcommand{\one}{\vec{1}}
\newcommand{\x}{\vec{x}}
\newcommand{\y}{\vec{y}}

\begin{document}

\begin{abstract}
  The set splittability problem is the following: given a finite collection of finite sets, does there exits a single set that contains exactly half the elements from each set in the collection? (If a set has odd size, we allow the floor or ceiling.) It is natural to study the set splittability problem in the context of combinatorial discrepancy theory and its applications, since a collection is splittable if and only if it has discrepancy $\leq1$.
  
  We introduce a natural generalization of splittability problem called the $p$-splittability problem, where we replace the fraction $\frac12$ in the definition with an arbitrary fraction $p\in(0,1)$. We first show that the $p$-splittability problem is NP-complete. We then give several criteria for $p$-splittability, including a complete characterization of $p$-splittability for three or fewer sets ($p$ arbitrary), and for four or fewer sets ($p=\frac12$). Finally we prove the asymptotic prevalence of splittability over unsplittability in an appropriate sense.
\end{abstract}

\maketitle

\section{Introduction}
\label{sec:intro}

Let $\B = \{B_1, \ldots, B_n\}$ be a collection of finite sets and let $0\leq p\leq1$. We say the collection $\B$ is \emph{$p$-splittable} if there exists a set $S$ (called a $p$-\emph{splitter}) such that for all $i\leq n$, we have that $|S \cap B_i|=\nint{p|B_i|}$, the nearest integer to $p|B_i|$. In the special case when $p=\frac12$, we will sometimes suppress the $p$ and use the terms $\B$ is \emph{splittable}, and $S$ is a \emph{splitter}.

Of course, the nearest integer $\nint{x}$ is not well-defined when $x$ is a half-integer; throughout the paper we adopt the convention that when $x$ is a half-integer, a statement about $\nint{x}$ is considered true if there is some choice of rounding which makes it true. Thus if $\nint{p|B_i|}$ is a half-integer, a $p$-splitter is allowed to satisfy either $|S \cap B_i|=\floor{p|B_i|}$ or $|S \cap B_i|=\ceiling{p|B_i|}$.

It is natural to study splittability and its generalizations in the context of combinatorial discrepancy theory. Given a collection $\B$ as above, the \emph{discrepancy} of $\B$ is
\[ \disc(\B) = \min_{S} \max_{i\leq n}\big| | B_i \cap S| - | B_i \setminus S| \big|\text{.}
\]
Intuitively, the discrepancy measures to what extent it is possible to simultaneously and evenly split each set in the collection. In fact $\disc(\B) \leq 1$ if and only if $\B$ is ($\frac{1}{2}$-) splittable.

In recent decades there have been many studies of upper bounds on the discrepancy of general and particular collections of sets. In 1981 Beck and Fiala showed that if every element of $\bigcup B$ is contained in at most $t$ of the sets in $\B$, then $\disc(\B)\leq2t-2$ \cite{beckfiala}. Incremental improvements to this bound can be found in works such as \cite{bednarchakHelm97, helm99, bukh13}. In 1985, Spencer gave an upper bound for the discrepancy of an arbitrary collection:
\[ \disc(\B) \leq 6 \sqrt{n}
\]
where $n$ is the number of sets \cite{sixstdev}.
Of course the discrepancy of any given collection may be much smaller than this bound, and since in most applications least discrepancy is best, it is natural to study the discrepancy $\leq1$ case.

Set splittability can also be viewed as a combinatorial version of the outcome of the ham sandwich theorem: given Lebesgue measurable subsets $B_1,\ldots,B_n\subset\R^n$ there exists a hyperplane $H$ such that for all $i\leq n$ exactly half the measure of $B_i$ lies to each side of $H$. If Lebesgue measure is replaced by an atomic measure, then some of the mass of $B_i$ may lie on $H$ itself. In this case the conclusion must be modified to say that that \emph{at most} half the measure of $B_i$ lies to each side of $H$ \cite{hamsandwichmayo}. Thus a ham sandwich hyperplane does not precisely solve the splittability problem, nor does set splittability help to find a geometric hyperplane, but the two problems are conceptually related.

A third way to think of set splittability is as a very strong form of hypergraph $2$-colorability. Recall that a hypergraph with hyperedges $B_1,\ldots,B_n$ is $2$-colorable if there exists a $\{\text{red},\text{blue}\}$-coloring of its vertices such that no hyperedge is monochromatic. With $p$-splittability we ask not simply that both colors are represented in each hyperedge, but that the color red always appears a prescribed percentage of the time.

In the next section we explore the computational complexity of $p$-splittability. In the case $p=\frac12$ it is known that the question of deciding whether a given collection is splittable is NP-complete. This follows from the fact that it is NP-hard to distinguish collections of discrepancy $0$ from collections of discrepancy $\Omega(\sqrt{n})$ \cite{charikar}. The significance of this result is that while there are randomized algorithms to find witnesses to Spencer's theorem \cite{bansal10,lovettMeka12}, in general even if a collection has discrepancy $o(\sqrt{n})$ one cannot expect to efficiently find a witness for this. In another related result, the problem of deciding whether a given hypergraph is $2$-colorable is NP-complete \cite{lovasz}. We will establish the corresponding hardness results in the case of $p$-splittability for arbitrary $p$. That is, we show that for any $0<p<1$, the $p$-splittability problem is NP-complete.

The fact that the $p$-splittability decision problem is hard means we do not expect to find a general and useful characterization of $p$-splittability. However it is possible to do so for small collections and for other special collections of sets. For an example involving (very) small collections, we will show that a collection $\B$ of at most two sets is $p$-splittable for any $p$. For an example involving special collections, suppose that $\B=\{B_1,\ldots,B_n\}$ is a collection of $n$ sets such that every element $x$ lies in exactly $n-1$ sets of $\B$. In this case we will show that $\B$ is $p$-splittable if and only if the sum $\nint{p|B_1|}+\cdots+\nint{p|B_n|}$ is divisible by $n-1$. (Note here our convention that half-integers may be rounded either up or down to make the condition hold.) The calculations used in these two results eventually lead us to a complete algebraic characterization of $p$-splittability for collections $\B$ of at most three sets.

If one specializes to the important case $p=\frac12$, some things become simpler and more characterizations become tractable. For example, if $\B=\{B_1,B_2,B_3\}$ is a collection of three sets then $\B$ is $\frac12$-unsplittable if and only if the sets $B_1\cap B_2\cap B_3^c$, $B_1\cap B_2\cap B_3^c$, $B_1\cap B_2\cap B_3^c$ are each odd in size and collectively cover $\bigcup\B$. (We will call these three sets the Venn regions of multiplicity $2$. This fact was previously observed in \cite{reu2014}.)

In this paper we state a complete characterization of the $\frac12$-unsplittable collections of four sets in terms of the sizes of its Venn regions. The characterization is substantially more complex than for three sets and involves more than ten delineated cases. To support the characterization we first prove a reduction lemma which implies that if $\B$ is unsplittable then it remains unsplittable after reducing the number of elements of each Venn region modulo $2$. We then carry out an exhaustive search for unsplittable configurations with a small number of elements. We remark that the exhaustive search was done with the aid of a computer program, and though it completed successfully we have not formally verified the correctness of the program.

The characterization of the $\frac12$-unsplittable configurations of four or fewer sets easily implies that unsplittability is extremely rare for collections of four or fewer sets. Although our method of finding unsplittable configurations becomes intractable for collections of five or more sets, we can show that this rarity phenomenon remains true. Specifically we show that if $n\to\infty$ and $k$ grows sufficiently fast relative to $n$, then the probability that a collection with $n$ sets and $k$ elements is splittable converges to $1$. In particular if $n$ is fixed and $k$ is large enough then most collections with $n$ sets and $k$ elements are splittable.

This paper is organized as follows. In Section~\ref{sec:complexity} we prove that the problem of deciding whether a given collection is $p$-splittable is NP-complete. In Section~\ref{sec:psplit}, we give criteria for deciding whether some special collections are $p$-splittable, and provide a complete characterization of $p$-splittability for collections of at most three sets. In Section~\ref{sec:12split}, we give further splittability criteria for the special case $p=\frac{1}{2}$ and use them to give a complete characterization of $\frac{1}{2}$-splittability for collections of at most four sets. Finally we show that for collections with sufficiently many elements, splittability is by far more common than unsplittability.

\textbf{Acknowledgement}. This article represents a portion of the research carried out during the 2016 math REU program at Boise State University. The program was supported by NSF grant \#DMS 1359425 and by Boise State University. We are grateful to several referees for valuable feedback and support.

\section{The complexity of $p$-splittability}
\label{sec:complexity}

In this section we establish that the $p$-splittability problem is NP-complete. Before proceeding, we set up some notation and clarify how we regard the $p$-splittability problem formally as a decision problem, which we denote $\psplit$.

To begin, if $\B=\{B_1,\ldots,B_n\}$ is a finite collection of subsets of $\{1,\ldots,m\}$ then the \emph{incidence matrix} of $\B$ is the $n\times m$ matrix $M$ whose $(i,j)$ entry is $1$ whenever $j\in B_i$, and $0$ otherwise. We will make significant use of the following notations surrounding incidence matrices: $M_i$ for the $i$th row of $M$; $\one$ for a vector of $1$'s whose length is determined by the context; $M_i\one$ for the number of $1$'s in the $i$th row of $M$, and; $\nint{pM\one}$ for the vector whose $i$th component is $\nint{pM_i\one}$).

Officially, an instance of $\psplit$ consists of a binary matrix $M$, which we think of as the incidence matrix of a collection $\B$. The matrix $M$ lies in $\psplit$ if and only if there exists a binary vector $\mathbf y$ such that $M\mathbf y=\nint{pM\one}$. Indeed, since $\nint{pM_i\one}=\nint{p|B_i|}$, we have that a binary vector $\mathbf y$ satisfies $M\mathbf y=\nint{pM\one}$ if and only if $\{i:y_i=1\}$ witnesses that $\B$ is $p$-splittable.

\begin{thm}
  \label{thm:np-complete}
  For any $0<p<1$, the problem $\psplit$ of determining whether a general collection is $p$-splittable is NP-complete.
\end{thm}

Note that $\psplit$ lies in NP because given an instance $M$ of $\psplit$ and a characteristic vector $\mathbf y$ of an ostensible splitter, one can easily decide in polynomial time (in the number of entries of $M$) whether $(M\mathbf y)_i$ is equal to $\nint{pM_i\one}$ for each $i$.

To establish that $\psplit$ is NP-complete, we will exhibit a polynomial-time reduction from the decision problem $\zoe$, which is known to be NP-complete, to $\psplit$. Here $\zoe$ stands for \emph{zero one equations}, and is formalized as follows. An instance of $\zoe$ consists of a binary matrix $A$. The matrix $A$ lies in $\zoe$ if and only if there exists a binary vector $\x$ such that $A\x=\one$. We note that $\zoe$ is similar to zero-one integer programming \cite{karp}, and its NP-completeness is established in \cite{zoe}.

In the definition of $\zoe$, we can assume without loss of generality that the matrix $A$ has no zero rows, since otherwise $A\x=\one$ is guaranteed to have no solution. We can further assume that at least one row of $A$ has at least two $1$'s, since otherwise $A\x=\one$ is guaranteed to have a solution, namely $\x=\one$.

Now in order to establish Theorem~\ref{thm:np-complete}, we will describe a mapping from binary matrices $A$ to binary matrices $M=M(A)$, with the property that $A$ lies in $\zoe$ if and only if $M$ lies in $\psplit$. In order to guarantee this, the matrix $M$ that we construct will have the special properties:
\begin{enumerate}
  \item $A$ is an upper-left submatrix of $M$;
  \item any solution $\x$ to $A\x=\one$ extends to a solution $\y$ to $M\y=\nint{pM\one}$; and
  \item any solution $\y$ to $M\y=\nint{pM\one}$ restricts to a solution $\x$ to $A\x=\one$ (or else its binary complement does; see below).
\end{enumerate}

Having described our general approach, we now proceed with the details.

\subsection{Specification of the construction}
\label{ssec:spec}

Let $A$ be a given $r\times c$ binary matrix, and let $p\in(0,1)$ be given. As mentioned above we may assume that for all $i$ we have $A_i\one>0$, and that for some $i$ we have $A_i\one>1$. We may further assume that $0<p\leq\frac12$ (without loss of generality as argued in Subsection~\ref{ssec:proof}). We construct a block matrix $M$ of the form:
\[M=
  \begin{bmatrix}
    A & B & C\\
    0 & D & E
  \end{bmatrix}.
\]
We now describe the blocks of $M$. Of course $A$ is the given matrix, and $0$ is a matrix of $0$'s of the appropriate dimensions. Before defining the rest of the blocks, we let $s=\max_i A_i\one$, and let $T$ be the set of indices of the $(s-1)$-many columns following the columns of $A$. Next we let $q=\frac{1-p}{p}$, and let $F$ be the indices of the $\max\left\{\ceiling{qs}-s+1,\ceiling{q}+1\right\}$-many columns to the right of the columns indexed by $T$. (The significance of these values will become more apparent in the example in the next subsection.)

We define $B$ to be an $r\times|T|$ matrix whose $i$th row contains $(A_i\one-1)$-many $1$'s, followed by all $0$'s. We define $C$ to be an $r\times|F|$ matrix whose $i$th row contains $(\ceiling{qA_i\one}-A_i\one+1)$-many $1$'s, followed by all $0$'s.

The blocks $D$ and $E$ each are built from smaller blocks. For $i\leq|T|$, let $D_i$ be the $\binom{|F|}{\ceiling{q}}\times|T|$ matrix whose $i$th column consists of $1$'s, and all other columns consist of $0$'s. And let $E_0$ denote a $\binom{|F|}{\ceiling{q}}\times|F|$ matrix whose rows consist of the indicator functions of the subsets of $\{1,\ldots,|F|\}$ of size $\ceiling{q}$. Then we let
\[D=\begin{bmatrix}D_1 \\ \vdots \\ D_{|T|}\end{bmatrix}
  \text{, and }
  E=\begin{bmatrix}E_0 \\ \vdots \\ E_0\end{bmatrix}
\]
Here there are $|T|$-many copies of $E_0$ in $E$.

It is easy to see that the dimensions of the matrix $M$ are polynomial in the dimensions of the matrix $A$. (Recall here that $p$ is a fixed parameter of the construction.) Hence the construction is polynomial time (in $r\cdot c$).

\subsection{Example of the construction}
\label{ssec:example}

Before proving that the construction satisfies our requirements, let us give an example. Suppose that $p=\frac{1}{3}$ and we are given the $\zoe$ system $A\mathbf{x}=\one$ given by
\[\begin{bmatrix}
    1 & 1 & 0 & 0\\
    1 & 1 & 1 & 0 
  \end{bmatrix}
  \mathbf{x}=
  \begin{bmatrix}
    1\\1
  \end{bmatrix}
\]
Then we have $A_1\one=2$, $A_2\one=3$, $s=\max_iA_i\one=3$, and $q= \frac{1-p}{p} = 2$. Thus $T$ consists of the $s-1=2$ columns $\{5,6\}$, and $F$ consists of the $\ceiling{qs}-s+1=4$ columns $\{7,8,9,10\}$.

The block $B$ is thus a $2\times2$ matrix with $A_1\one-1=1$-many $1$'s in the first row and $A_2\one-1=2$-many $1$'s in the second row. The block $C$ is a $2\times4$ matrix with $\ceiling{qA_1\one}-A_1\one+1=3$-many $1$'s in the first row and $\ceiling{qA_2\one}-A_2\one+1=4$-many $1$'s in the second row.

Next, the blocks $D_1$ and $D_2$ are each $\binom{|F|}{\ceiling{q}} \times |T|$ which comes to $6\times2$. Block $D_1$ is a column of $1$'s followed by a column of $0$'s, and block $D_2$ is a column of $0$'s followed by a column of $1$'s.

Finally, the block $E_0$ is $\binom{|F|}{\ceiling{q}}\times|F|$ which comes to $6\times4$. The $6$ rows of $E_0$ correspond to the $6$ subsets of $\{1,\ldots,|F|\}=\{1,\ldots,4\}$ of size $\ceiling{q}=2$. The full matrix $M$ and the corresponding system $M\y=\nint{pM\one}$ are displayed in Figure~\ref{fig:example-construction}.

\begin{figure}[h]
  \[\begin{bmatrix}
      1 & 1 &   &   & \vline & 1 &   & \vline & 1 & 1 & 1 &  \\ 
      1 & 1 & 1 &   & \vline & 1 & 1 & \vline & 1 & 1 & 1 & 1\\\hline
        &   &   &   & \vline & 1 &   & \vline & 1 & 1 &   &  \\
        &   &   &   & \vline & 1 &   & \vline & 1 &   & 1 &  \\
        &   &   &   & \vline & 1 &   & \vline & 1 &   &   & 1\\
        &   &   &   & \vline & 1 &   & \vline &   & 1 & 1 &  \\
        &   &   &   & \vline & 1 &   & \vline &   & 1 &   & 1\\
        &   &   &   & \vline & 1 &   & \vline &   &   & 1 & 1\\
        &   &   &   & \vline &   & 1 & \vline & 1 & 1 &   &  \\
        &   &   &   & \vline &   & 1 & \vline & 1 &   & 1 &  \\
        &   &   &   & \vline &   & 1 & \vline & 1 &   &   & 1\\
        &   &   &   & \vline &   & 1 & \vline &   & 1 & 1 &  \\
        &   &   &   & \vline &   & 1 & \vline &   & 1 &   & 1\\
        &   &   &   & \vline &   & 1 & \vline &   &   & 1 & 1\\
    \end{bmatrix}
    \y=
    \begin{bmatrix}
      2\\3\\1\\1\\1\\1\\1\\1\\1\\1\\1\\1\\1\\1
    \end{bmatrix}
  \]
  \caption{The system $M\y=\nint{pM\one}$ constructed in our example.\label{fig:example-construction}}
\end{figure}

Having given the example construction, we briefly preview how the proof will play out in this case. The matrix $M$ is the incidence matrix for the collection $\{B_1,\ldots,B_{14}\}$ of subsets of $\{1,2,\ldots,10\}$, where the characteristic vector of $B_i$ is the $i$th row of $M$. The collection is splittable if and only if the system $M\textbf{y}=\nint{pM\one}$ has a binary solution. In our example the values on the right-hand side are $pM_1\one=2$, $pM_2\one=3$, and $pM_i\one=1$ for $i>2$.

Note that if $A\x=\one$, then $\x$ extends to a solution of $M\y=\nint{pM\one}$ by setting the components with indices in $T$ to be $1$ and the components with indices in $F$ to be $0$. Conversely, if $M\y=\nint{pM\one}$ then the components of $\y$ with indices in $T$ are forced to be $1$ and the components of $\y$ with indices in $F$ are forced to be $0$. Indeed, this can be seen by inspecting the $D$ and $E$ blocks of $M$, and is proved formally in the next subsection. Finally since the rows of the $B$ block have exactly one less $1$ than the rows of $A$, and since this number is also one less than the corresponding component on the right-hand side, such a $\y$ must restrict to a solution to $A\mathbf x=\one$.

\subsection{Proof of the main theorem}
\label{ssec:proof}

We now establish that the construction described above is indeed a reduction from $\zoe$ to $\psplit$. We will assume throughout that $0<p\leq\half$, since a collection is $p$-splittable if and only if it is $(1-p)$-splittable (simply take the complement of the witnessing splitter). To begin, we present a simple rounding calculation that will be used below.

\begin{lem}
  \label{lem:rounding}
  Let $0 < p \leq \half$ and $q=\frac{1-p}{p}$ as before. Then for any $m\in\N$ we have
  \[\nint{p(m+\ceiling{qm})}=m.
  \]
\end{lem}

\begin{proof}
  Let $\varepsilon = \ceiling{qm}-qm$. Then
  \begin{align*} 
    \nint{ p ( m + \ceiling{qm} ) }
  &= \nint{p(m+qm+\varepsilon)} \\
  &= \nint{pm+m(1-p)+p\varepsilon} \\
  &= \nint{m+p\varepsilon} \\
  &= m+\nint{p\varepsilon }
\end{align*}
Since $0\leq\varepsilon<1$ and $p\leq\half$, we have $p\varepsilon<\half$, which gives that the last quantity equals $m$ as desired.
\end{proof}

Next we calculate the values on the right-hand side of the system $M\y=\nint{pM\one}$.

\begin{lem}
  \label{lem:b}
  Let $0<p\leq\frac12$, let $A$ be an $r\times c$ matrix, and let $M$ be constructed as above. Then $\nint{pM_i\one}=A_i\one$ for all $i\leq r$, and $\nint{pM_i\one}=1$ for all $i>r$.
\end{lem}

\begin{proof}
  First consider $i\leq r$. Then
  \begin{align*}
    \nint{pM_i\one}
    &=\nint{p ( A_i \one +B_i \one +C_i \one )}\\
    &=\nint{p\Big( A_i \one + \left(A_i \one -1\right) + \big(\ceiling{qA_i\one}-A_i \one +1 \big) \Big)}\\
    &=\nint{p( A_i \one + \ceiling{qA_i\one}) }
  \end{align*}
  By Lemma~\ref{lem:rounding}, the latter quantity is simply $A_i\one$, as claimed.

  Next consider $i>r$. Here we have
  \[\nint{p M_i \one} = \nint{ p( D_i \one + E_i \one)} = \nint{ p ( 1 + \ceiling{q} )}.
  \]
  Again using Lemma~\ref{lem:rounding}, the latter quantity is $1$, as desired.
\end{proof}

To commence with the body of the proof, we first show that if $A\x=\one$ has a solution, then the system $(M\y)_i=\nint{pM_i\one}$ has a solution. Given a solution $\x$ to $A\x=\one$, we extend $\x$ to a vector $\y$ by appending $|T|$-many $1$'s followed by $|F|$-many $0$'s. Then for $i\leq r$ we have
\[M_i\y = A_i \x + B_i \one + C_i \vec{0} = 1 + (A_i \one - 1) = A_i \one
\]
By Lemma~\ref{lem:b}, this is equal to $\nint{pM_i\one}$, as desired. On the other hand, for $i > r$ we have
\[M_i\y = D_i \one + E_i \vec{0} = D_i \one = 1
\]
(Here $D_i,E_i$ denote the $i$th row of $D,E$, not the $i$th block.) Again by Lemma~\ref{lem:b}, this is equal to $\nint{pM_i\one}$, as desired.

For the converse, we show that if $M\y=\nint{pM\one}$ has a solution then $A\x=\one$ has a solution. We make a series of claims about the structure of the solution $\y$ that will enable us to create from it a solution $\x$ to $A\x=\one$.

\begin{claim}
  \label{claim:disjoint}
  It is not the case that there are $j \in T$ and $k\in F$ such that both $y_j = 1$ and $y_k = 1$.
\end{claim}

\begin{claimproof}
  Suppose towards a contradiction that $j \in T$, $k \in F$ and $y_j = y_k = 1$. Recalling the definitions of $D$ and $E$, we can find a row index $i>r$ such that $M_i$ has a $1$ in its $j$th and $k$th columns. It follows that $M_i\y\geq2$, which contradicts the calculation from Lemma~\ref{lem:b} that $M_i\y=\nint{pM_i\one}=1$.
\end{claimproof}

\begin{claim}
  \label{claim:TF}
  If $0 < p < \half$, then for all indices $i \in T$, $y_i = 1$ and for all indices $j \in F$, $y_j = 0$.
\end{claim}

\begin{claimproof}
  Suppose towards a contradiction that there is a $j \in F$ such $y_j = 1$. By the previous claim, for all $i\in T$ we have $y_i=0$. If there is just one such $j\in F$ with $y_j=1$, then by construction of $E_0$ we can find a row $\ell>r$ such that the $j$th entry of $M_\ell$ is $0$. This implies that $M_\ell\y=0$, contradicting that $M_\ell\y=\nint{pM_\ell\one}=1$.

  On the other hand if there are two distinct $j,j'\in F$ with $y_j=y_{j'}=1$, then since $p<\half$ implies $q>1$, we can find a row $\ell>r$ such that the $j$th and $j'$th entries of $M_\ell$ are both $1$. This implies that $M_\ell\y\geq2$, again contradicting that $M_\ell\y=1$.

  Thus we have shown that $y_j=0$ for all $j\in F$. It follows from the construction of $D$ that $y_i=1$ for all $i\in T$.
\end{claimproof}

To continue the proof, assume first that $p<\half$. Then by Claim~\ref{claim:TF}, for all $i\in T$ we have $y_i=1$ and for all $j\in F$ we have $y_j=0$. Letting $\x$ denote the restriction of $\y$ to its first $c$ entries, for any $i\leq r$ we have $(M\y)_i=(A\x)_i+(A\one)_i-1$. By Lemma~\ref{lem:b} we also know that $(M\y)_i=(A\one)_i$. It follows that $A\x=\one$.

Next consider the case when $p=\half$. Then $q=1$ so both $D$ and $E$ have exactly one $1$ per row. It follows from Claim~\ref{claim:disjoint} that we either have
\begin{enumerate}
  \renewcommand{\theenumi}{(\Roman{enumi})}
  \renewcommand{\labelenumi}{(\Roman{enumi}).}
  \item $y_i = 1$ for all $i \in T$ and $y_j = 0$ for all $j \in F$, or \label{it:selectT}
  \item $y_i = 0$ for all $i \in T$ and $y_j = 1$ for all $j \in F$. \label{it:selectF}
\end{enumerate}
If \ref{it:selectT} holds, we can do as we did when $p < \half$, so we are done. Otherwise, if \ref{it:selectF} holds, let $\y' = \one - \y$. Then
\[ M_i\y' = M_i( \one - \y) = M_i \one - \nint{M_i \one p}. \]
We know that $A_i \one$ has the opposite parity of $A_i \one - 1 = B_i \one$, and that $C_i \one = 1$ when $p = \half$. Therefore, $M_i \one$ is even, so $M_i \one - \nint{M_i \one p} = \nint{M_i \one p}$, meaning that $M\y'=\nint{pM\one}$.

Thus, $\y'$ also corresponds to a valid splitter of $\B$, and since $y_i' = 1$ for all $i \in T$ we must also have that its first $c$ entries pick out exactly one $1$ per row of $A$ by the same argument as the case $p<\half$.  Therefore, taking $\x$ to be restriction of $\y'$ to its first $c$ entries, we once again have that $A\x=\one$.

This concludes the proof that the construction from $A$ of $M$ is polynomial time (in $r\cdot c$) reduction from $\zoe$ to $\psplit$.

\section{$p$-Splittability criteria and characterizations}
\label{sec:psplit}

The result of the previous section implies that it is hard (assuming $NP\neq P$) to find a general characterization of $p$-splittability. Nevertheless, in this section we provide several $p$-splittability criteria for special types of collections. Furthermore we completely characterize $p$-splittability for collections of at most three sets.

Before we begin our study, it is useful to introduce the following notation. For a collection $\B=\{B_1,\ldots,B_n\}$ and an element $x$, the \emph{multiplicity} of $x$ is the number $m_x$ of sets $B_i$ such that $x\in B_i$. Given a subcollection $\{B_{i_1}, \ldots, B_{i_k}\} \subset \B$, we define the associated \emph{Venn region} of $\B$ to be the set of elements $x$ that lie in precisely the sets $B_{i_1}, \ldots, B_{i_k}$ and in no other $B_j$. Venn regions corresond pictorially to regions of the Venn diagram formed by the sets $B_1, \ldots B_n$ (we shall illustrate this later in Figure~\ref{fig:venn-regions}). If $R$ is a Venn region associated to a subcollection of cardinality $m$, then all elements of $R$ have multiplicity $m$, so we also say that $R$ has multiplicity $m$.

In the following result, we will say that a sequence $\{t_i\}$ is a \emph{target sequence} for $\B$ if $0\leq t_i\leq|B_i|$ for $1 \leq i \leq n$. The target sequence $t_i$ is \emph{achievable} if there is a set $S$ such that $|S\cap B_i|=t_i$ for all $i$.

\begin{lem}
  \label{lem:parity}
  Let $\B = \{B_1, \ldots, B_n\}$ be a collection of sets and assume that for every $x\in\bigcup\B$ the multiplicity $m_x$ is divisible by $m$. If the target sequence $\{t_i\}$ is achievable, then $\sum t_i$ is divisible by $m$.
\end{lem}
  
\begin{proof}
    Let $S$ be a set witnessing that $\{t_i\}$ is achievable. Then
  \[\sum_{1 \leq i\leq n} t_i=\sum_{1 \leq i\leq n}|S\cap B_i|=\sum_{x\in S}m_x \]
  By hypothesis, $m_x$ is divisible by $m$ for all $x$, so the right-hand side is also divisible by $m$.
\end{proof}

Since $\B$ is $p$-splittable if and only if the target sequence $t_i = \nint{p|B_i|}$ is achievable, the contrapositive of Lemma~\ref{lem:parity} provides a useful condition for showing that certain collections are not $p$-splittable. The converse of Lemma~\ref{lem:parity} is false in general: for a counterexample, take $\B = \{ \{1,2,3\}, \{1,4,6\}, \{2,5,6\}, \{3,4,5\} \}$. The multiplicity of every element of $\bigcup\B$ is 2, which divides the sum of the elements of the target sequence $\{3,3,3,1\}$; however, any set that contains all three elements of three of the sets in $\B$ contains all elements of $\bigcup\B$, so $\{3,3,3,1\}$ is not achievable.
Despite this, the converse of Lemma~\ref{lem:parity} does hold in the following very special case.

\begin{lem}
  \label{lem:n-1}
  Let $\B = \{B_1,\ldots, B_n\}$ be a collection such that for all $x \in \bigcup \B$, $m_x = n-1$. If $\sum \nint{p|B_i|}$ is divisible by $n-1$, then $\B$ is $p$-splittable.
\end{lem}

\begin{proof}
  Let $b_i = \left|\bigcup \B - B_i \right|$ be the size of the (unique) Venn region of multiplicity $n-1$ which is not contained in $B_i$. Further let $t_i=\nint{p|B_i|}$ be the target sequence. We wish to find values $\bar b_i$ such that $0\leq\bar b_i\leq b_i$ and $\sum_{j\neq i}\bar b_j=t_i$ for each $i$. Indeed, then we would be able to form a $p$-splitter by selecting $\bar b_i$ elements from each region $\bigcup \B - B_i$, and combining these elements into one set.
  
  To find the $\bar b_i$, the coefficient matrix of the system $\sum_{j\neq i}\bar b_j=t_i$ is square and invertible, and an elementary calculation using Gaussian elimination yields the unique solution:
  \begin{equation}
    \label{eq:barbi}
    \bar b_i=\frac{1}{n-1} \left(-(n-2)t_i+\sum_{j\neq i} t_j\right)\text{.}
  \end{equation}
  Note that $\bar b_i$ is always an integer, because the above expression is equal to 
  \[ \frac{1}{n-1}\left(-(n-1) t_i+\sum_j t_j\right) = -t_i + \frac{1}{n-1} \sum_{j} t_j \]
  and $\sum t_j$ is divisible by $n-1$ by hypothesis. Hence it remains only to establish that $0\leq\bar b_i\leq b_i$.

  For this, note that $t_i=\nint{p|B_i|} = \nint{\sum_{j\neq i}pb_j} = \epsilon_i+\sum_{j\neq i}pb_j$ where $|\epsilon_i|\leq\frac12$. Substituting this expression in for every $t_i$ in Equation~\eqref{eq:barbi}, we note that $pb_i$ occurs $n-1$ times in the parentheses while all other $pb_j$ occur $n-2$ times negatively and $n-2$ times positively. Thus, the $p b_j$ cancel, leaving us just with $pb_i$ and error terms as follows:
  \begin{equation}
    \label{eq:barbi2}
    \bar b_i=pb_i+\frac{1}{n-1}\left(-(n-2)\epsilon_i+\sum_{j\neq i}\epsilon_j\right)\text{.}
  \end{equation}
  There are $2n-3$ many $\epsilon$ terms in the parentheses, so we can conclude that $\bar b_i=pb_i+E$ where $|E|<1$. Since $0\leq pb_i\leq b_i$ and all of $0,b_i,\bar b_i$ are integers, it follows that $0\leq\bar b_i\leq b_i$ too.
\end{proof}

We are now ready to begin our classification of $p$-splittability for collections of size $\leq 3$. We begin with the simple case of just two sets, because it helps motivate some of the steps for the three-set case below.

\begin{thm}
  Every collection of two sets is $p$-splittable for every $0 \leq p \leq 1$.
\end{thm}

\begin{proof}
  Let $\B = \{B_1, B_2\}$ be a given two-set collection. Replacing $p$ with $1-p$ if necessary, suppose that $p\leq\frac12$. Fix the following notation for the sizes of the regions of $\B$: $a_1=|B_1\cap B_2^c|$, $a_2=|B_1^c\cap B_2|$, and $b=|B_1\cap B_2|$ (see Figure~\ref{fig:venn-regions}). Next let $t_1=\nint{p(a_1+b)}$ and $t_2=\nint{p(a_2+b)}$ denote the target cardinalities for $S\cap B_1$ and $S\cap B_2$ for a $p$-splitter $S$. To show $\B$ is $p$-splittable it suffices to find integers $\bar a_i,\bar b$ such that (i) $0\leq\bar a_i\leq a_i$, (ii) $0\leq\bar b\leq b$, and (iii) $\bar a_i+\bar b=t_i$.
  
  \begin{figure}
    \begin{tikzpicture}
      \draw (0,0) circle (1);
      \draw (1,0) circle (1);
      \node[anchor=east] at (-1,0) {$B_1$};
      \node[anchor=west] at (2,0) {$B_2$};
      \node at (-.5,0) {$a_1$};
      \node at (1.5,0) {$a_2$};
      \node at (.5,0) {$b$};
    \end{tikzpicture}
    \hspace{.5in}
    \begin{tikzpicture}[rotate=120]
      \draw (90:.6) circle (1);
      \draw (210:.6) circle (1);
      \draw (330:.6) circle (1);
      \node at (90:1.9) {$B_1$};
      \node at (210:1.9) {$B_2$};
      \node at (330:1.9) {$B_3$};
      \node at (90:1) {$a_1$};
      \node at (210:1) {$a_2$};
      \node at (330:1) {$a_3$};
      \node at (-90:.7) {$b_1$};
      \node at (-210:.7) {$b_3$};
      \node at (-330:.7) {$b_2$};
      \node at (0:0) {$c$};
    \end{tikzpicture}
    \caption{At left: sizes of the Venn regions of a two-element collection $\{B_1,B_2\}$. At right: sizes of the Venn regions of a three-element collection $\{B_1,B_2,B_3\}$.\label{fig:venn-regions}}
  \end{figure}

  For this let $\bar b=\nint{pb}$ and $\bar a_i=t_i-\bar b$ so that (ii) and (iii) are clearly satisfied. Of course the definitions of both $t_i$ and $\bar b$ may be ambiguous; in such cases we ensure that ($\star$) if $a_i=0$ then we choose $\bar a_i=0$ too. (This is always possible by choosing the same rounding for both $\bar b$ and $t_i$.)

  To see that (i) is satisfied, write $\varepsilon=\bar b-pb$ for the rounding error in computing $\bar b$ and $\epsilon_i=t_i-p(a_i+b)$ for the rounding error in computing $t_i$. Then the definitions of $\bar b$ and $\bar a_i$ easily imply that
  \[\bar a_i=pa_i+(\epsilon_i-\varepsilon)\text{.}
  \]
  Since $|\epsilon_i|\leq\frac12$ and $|\varepsilon|\leq\frac12$ we know that $|\epsilon_i-\varepsilon|\leq1$. Assuming $a_i>0$ the above equation gives $-1<\bar a_i<a_i+1$, and since $\bar a_i$ and $a_i$ are integers, (i) is satisfied. On the other hand if $a_i=0$ then by assumption ($\star$) we have $\bar a_i=0$ too, so (i) is clearly satisfied.
\end{proof}

To state our results for three sets, we extend the notation from the previous proof. For a three-set collection $\B=\{B_1,B_2,B_3\}$ we let $a_i$ denote the number of multiplicity~1 elements of $B_i$, let $b_i$ denote the number of multiplicity~$2$ elements \emph{not} in $B_i$, and let $c$ denote the number of multiplicity~3 elements (see Figure~\ref{fig:venn-regions}). As in the previous proof we let $t_i=\nint{p|B_i|}$ be the targets and $\epsilon_i=t_i-p|B_i|$ be the rounding error. Finally we set the values $\rho_i=-\epsilon_i+\sum_{j\neq i}\epsilon_j$.

\begin{lem}
  \label{lem:psplit-core}
  Assume that $p\leq\frac12$, and let $\B=\{B_1,B_2,B_3\}$ be given. Also assume there are no multiplicity~1 elements, that is, all $a_i=0$. Then $\B$ is not $p$-splittable if and only if $\sum t_i$ is odd and at least one of the conditions holds:
  \begin{enumerate}
  \item $c=0$; or
  \item $pc<\frac12$ and some $pb_i+\frac12(pc-1+\rho_i)<0$.
  \end{enumerate}
\end{lem}

\begin{proof}
  First observe that $\B$ is $p$-splittable if and only if one can find values $\bar b_i$ and $\bar c$ such that $0\leq \bar b_i\leq b$, $0\leq\bar c\leq c$, and $\bar c+\sum_{j\neq i}\bar b_i=t_i$. Indeed, such values of $\bar b_i$ and $\bar c$ correspond to the number of elements of the corresponding Venn regions needed to make a $p$-splitter.
  
  Assuming one has chosen a value for $\bar c$, we can again use Gaussian elimination to find that the system of equations $\bar c+\sum_{j\neq i}\bar b_i=t_i$ has the unique solution:
  \begin{equation}
    \label{eq:psplit-initial-solution}
    \bar b_i=\frac12\left(-t_i+\sum_{j\neq i}t_j-\bar c\right)\text{.}
  \end{equation}
  At this point we can observe that in order to achieve integer values of $\bar b_i$, one must choose the value $\bar c$ to have the same parity as $\sum t_j$. Next we substitute $t_j=pc+p\sum_{k\neq j}b_k+\epsilon_j$ to rewrite the above equation as
  \begin{equation}
    \label{eq:psplit-solution}
    \bar b_i=pb_i+\frac12\left(pc-\bar c+\rho_i\right)\text{.}
  \end{equation}
  
  Now assume that $\sum t_i$ is odd and that (a) or (b) holds. If (a) holds, then Lemma~\ref{lem:parity} demonstrates that $\B$ is unsplittable. Thus assume that (b) holds. Since $\sum t_i$ is odd, we cannot choose $\bar c$ to be of even parity and in particular cannot choose $\bar c=0$. Condition~(b) together with Equation~\eqref{eq:psplit-solution} implies that any positive value of $\bar c$ results in $\bar b_i<0$. Thus $\B$ is once again unsplittable.

  For the converse we will need to show that if either $\sum t_i$ is even or both (a) and (b) are false, then $\B$ is splittable.
  
  \begin{claim}
    \label{claim:pc}
    The choice $\bar c=\nint{pc}$ always ensures that $0\leq\bar b_i\leq b_i$.
  \end{claim}
    
  \begin{claimproof}
    Note first that this choice implies $|pc-\bar c|\leq\frac12$. Moreover we can always assume $|\rho_i|<\frac32$, since otherwise all $|\epsilon_j|=\frac12$ and we would be able to change the rounding of the targets $t_i$ (even while preserving the parity of $\sum t_i$). Thus Equation~\eqref{eq:psplit-solution} implies that $\bar b_i=pb_i+E$ where $E<1$, and we can therefore argue as in the proof of Lemma~\ref{lem:n-1} to complete the claim.
  \end{claimproof}

  Of course we cannot necessarily choose $\bar c=\nint{pc}$, since this may not have the same parity as $\sum t_j$. Thus we need the following.
  
  \begin{claim}
    \label{claim:c+-}
    One of the two choices $\bar c_-:=\nint{pc}-1$ or $\bar c_+:=\nint{pc}+1$ ensures that $0\leq\bar b_i\leq b_i$.
  \end{claim}
  
  \begin{claimproof}
    The two choices result in values $pc-\bar c_-$ and $pc-\bar c_+$. These values differ by $2$ and have absolute value $\leq\frac32$. Meanwhile the $\rho_i$ lie in some interval of length $\leq2$ which is contained in $(-\frac32,\frac32)$. It is straightforward to conclude that either all $|pc-\bar c_-+\rho_i|<2$ or all $|pc-\bar c_++\rho_i|<2$. Thus one of two choices $\bar c=\bar c_-$ or $\bar c_+$ gives values $\bar b_i=pb_i+E'$ where $E'<1$, and we are again done as in the previous claim.
  \end{claimproof}

  Now assume that $\sum t_i$ is even. If $\nint{pc}$ is even then by Claim~\ref{claim:pc} we have that $\bar c=\nint{pc}$ leads to a solution. And if $\nint{pc}$ is odd then we always have $0\leq\nint{pc}\pm1\leq c$ (here we are using $p\leq\frac12$). Thus by Claim~\ref{claim:c+-} one of the choices $\bar c=\nint{pc}\pm1$ leads to a solution as well.

  Finally assume both (a) and (b) are false. Since (a) is false and $p\leq\frac12$ we have that $\nint{pc}+1\leq c$. On the other hand since (b) is false we either have (i) $pc\geq\frac12$, or else (ii) $pc<\frac12$ and all $pb_i+\frac12(pc-1+\rho_i)\geq0$. In case (i) we have $0\leq\nint{pc}\pm1\leq c$, which we have previously shown implies $\B$ is splittable. In case (ii) we have $\nint{pc}+1=1$ and moreover that the choice $\bar c=1$ leads to a valid solution for all $\bar b_i$. This concludes the proof.
\end{proof}

We remark that in the previous lemma, if any of the $p|B_i|$ is a half-integer, then $\B$ is $p$-splittable. Indeed, in this case we can select the set target $t_i$ to make $\sum t_i$ even. We also note that if case (b) of the lemma holds, then it is not difficult to see there is a unique $i$ such that $2pb_i+pc-1+\rho_i<0$, and moreover that $\rho_i<-1$, that $\epsilon_i>0$, and that the other two $\epsilon_j<0$.

In the next result we consider the case when a collection $\B$ of three sets has elements of multiplicity~$1$.

\begin{lem}
  \label{lem:psplit-monofolds}
  Assume that $p\leq\frac12$, and let $\B=\{B_1,B_2,B_3\}$ be a given collection with $a_i,b_i,c$ as in Figure~\ref{fig:venn-regions}. Then provided at least one of the $a_i$ is sufficiently large, $\B$ is $p$-splittable.
\end{lem}

\begin{proof}
  First let $\B^{(0)}$ be the collection obtained from $\B$ by removing all elements of multiplicity~$1$. For the rest of the proof, let $b_i,c,t_i,\rho_i$ be as in the previous lemma, \emph{for the collection} $\B^{(0)}$.

  Suppose first that $\B^{(0)}$ is $p$-splittable. Then $\B$ is splittable too; in fact we can show that any splittable collection remains splittable after adding elements of multiplicity~$1$. To see this, it suffices to show it when we add just one element $a$ of multiplicity~$1$ to some set $B_i$. Now if adding $a$ raises the value of $t_i$ by $1$, then we include $a$ in the splitter; otherwise we would exclude $a$ from the splitter.

  Next suppose that $\B^{(0)}$ is $p$-unsplittable. Then by Equation~\eqref{eq:psplit-initial-solution} we can ``split'' the collection $\B^{(0)}$ by setting $\bar c=0$ and $\bar b_i=\frac12\left(-t_i+\sum_{j\neq i}t_j\right)$. However these choices of $\bar b_i$ will be half-integers, and need not satisfy $0\leq\bar b_i\leq b_i$.

  \begin{claim}
    \label{claim:barbi}
    For all $i$ we have $-\frac12\leq\bar b_i\leq b_i+\frac12$. Moreover there is at most one $i$ such that $\bar b_i=-\frac12$ or $\bar b_i=b_i+\frac12$.
  \end{claim}
  
  \begin{claimproof}  
    The first statement follows directly from Equation~\eqref{eq:psplit-solution}, together with $\bar c=0$ and $pc<\frac12$.

    For the second statement, our definition of $\bar b_i$ implies that the only possible contrary case is when two of the $\bar b_i$ have errors at opposite extremes, say $\bar b_2=-\frac12$ and $\bar b_3=b_3+\frac12$. We now show this implies that $b_3=0$. Indeed, Equation~\eqref{eq:psplit-solution} for $\bar b_2$ says $-\frac12=pb_2+\frac12(pc+\rho_2)$ and this implies $\rho_2<-1$. Since all $\rho_i$ lie within an interval of length $2$, it follows that $\rho_3<1$. Then Equation~\eqref{eq:psplit-solution} for $\bar b_3$ says $b_3+\frac12=pb_3+\frac12(pc+\rho_3)$. Now $p\leq\frac12$, $pc<\frac12$, and $\rho_3<1$ all together imply $b_3=0$.

    Using $b_3=0$, we obtain in particular that $t_3\geq t_2$. On the other hand Equation~\eqref{eq:psplit-initial-solution} for $\bar b_2$ says $-1=t_1-t_2+t_3$, and this implies $t_2>t_3$. This is a contradiction, and completes the proof of the claim.
  \end{claimproof}

  Now we can finish the proof as follows. Let $\B^{(1)}$ be the collection obtained from $\B$ by removing just the elements of $B_2$ and $B_3$ of multiplicity~$1$. That is, $\B^{(1)}$ is obtained by zeroing out $a_2$ and $a_3$. We will show that if $a_1$ is sufficiently large, then $\B^{(1)}$ is splittable.

  For this, we will use the notation $t_1^{(1)}$ for the target value of $B_1$ as it would be defined for $\B^{(1)}$ (or equivalently for $\B$). Thus in particular $t_1^{(1)}\geq t_1$. In the next paragraph we will only have need of small values of $a_1$, so that we need only consider the cases $t_1^{(1)}=t_1$ and $t_1^{(1)}=t_1+1$.

  Claim~\ref{claim:barbi} implies that at least one of the two triples $\{\bar b_1-\frac12,\bar b_2+\frac12,\bar b_3+\frac12\}$ or $\{\bar b_1+\frac 12,\bar b_2-\frac12,\bar b_3-\frac12\}$ lies within the desired bounds $[0,b_1],[0,b_2],[0,b_3]$. In the first case if $a_1$ is large enough that $t_1^{(1)}=t_1+1$ then the triple yields a valid splitting. In the second case if $a_1=1$ and $t_1^{(1)}=t_1$ then the triple extends to a valid splitting by selecting the single element of $a_1$. And if $a_1\geq2$ and $t_1^{(1)}=t_1+1$ then the triple extends to a valid splitting by selecting two elements from $a_1$.

  Thus we have shown in each case that there exists a value of $a_1$ that results in $\B^{(1)}$ being splittable. By the argument from the second paragraph, any larger value of $a_1$ will also result in $\B^{(1)}$ being splittable. Again using the argument from the second paragraph, this always implies $\B$ is splittable.
\end{proof}

The above lemma may seem natural, since intuitively the presence of elements of multiplicity~$1$ makes it easier to find a splitter. However the analogous result is false for collections of four or more sets. Indeed if $\{B_1,B_2,B_3\}$ is an unsplittable collection, then we can create unsplittable collections $\{B_1,B_2,B_3,B_4\}$ where the Venn region of multiplicity~$1$ given by $B_1^c\cap B_2^c\cap B_3^c\cap B_4$ is as large as we like.


\section{$\frac12$-splittability criteria and characterizations}
\label{sec:12split}

In the previous section, we examined $p$-splittability criteria for arbitrary $p$. In this section we specialize to the important case $p=\frac12$. After providing another very general lemma, we use it to give a complete characterization of splittability for collections of at most four sets, under the assumption that an exhaustive search by computer is implemented correctly. Throughout this section, the term splittability will always refer to $\frac12$-splittability.

The following result, while quite simple, is useful for converting our understanding of collections with few elements into more general theorems.

\begin{lem}[Reduction Lemma]
  \label{lem:reduction}
  Let $\B$ be a given collection, and let $\B'$ be a collection obtained from $\B$ by adding an even number of elements to any of its Venn regions. Then if $\B$ is splittable, so is $\B'$.
\end{lem}

\begin{proof}
  If $S$ is a splitter for $\B$, then we can construct a splitter $S'$ for $\B'$ as follows. Begin by putting all the elements of $S$ into $S'$. Then for each Venn region $R$ of $\B$ and corresponding Venn region $R'$ of $\B'$, put half of the elements of $R'\setminus R$ into $S'$. It is easy to see that $S'$ is a splitter for $\B'$.
\end{proof}

Before stating our characterization of splittability for configurations with four sets, we review the known characterization of splittability for configurations with three or fewer sets (see \cite{reu2014}).

\begin{prop}
  \label{prop:oddoddodd}
  Any collection of one or two sets is splittable. A collection of three sets is unsplittable if and only if both:
  \begin{enumerate}
  \item every Venn region of multiplicity~$2$ contains an odd number of elements; and
  \item all other Venn regions are empty.
  \end{enumerate}
\end{prop}

The proposition can easily be extracted from the results of the previous section. It is also possible to give a simple and direct proof, as was done in \cite{reu2014}.

Next we will state our characterization of splittability for collections of four sets. In order to do so we will need to work with four-set Venn diagrams, shown as four-lobed ``hearts'' with each lobe representing one set. The diagram below shows four of the diagrams; one with each of the four sets shaded.

\begin{center}
    \begin{tikzpicture}[scale=0.5]
	\newcommand{\lobe}{
	    \draw (-1,3) -- (-1,0) -- (1,0) -- (1,3) ;
	    \draw (1,3) arc[start angle=0, end angle=180, radius=1] ;
	}
	\newcommand{\doublelobe}{
	    \lobe
	    \begin{scope}[xshift=1cm]
		\lobe
	    \end{scope}
	}
	\begin{scope}[xshift=3cm, rotate around = {45:(-1,0)}]
	    \begin{scope}[opacity=0.4]
		\filldraw (-1,3) -- (-1,0) -- (1,0) -- (1,3) ;
		\filldraw (1,3) arc[start angle=0, end angle=180, radius=1] ;
	    \end{scope}
	    \doublelobe
	\end{scope}
	\begin{scope}[rotate around = {-45:(2,0)}]
	    \doublelobe
	\end{scope}
	\begin{scope}[yshift=0.707cm, rotate around={-45:(2,0)}, transform shape]
	    \foreach \x in {-1, 0,1,2} {
		\foreach \y in {0,1,2,3} {
		    \draw (\x,\y) node(\x\y) { } ; 
		}
	    }
	\end{scope}
    \end{tikzpicture}
    \hspace{1cm}
    \begin{tikzpicture}[scale=0.5]
	\newcommand{\lobe}{
	    \draw (-1,3) -- (-1,0) -- (1,0) -- (1,3) ;
	    \draw (1,3) arc[start angle=0, end angle=180, radius=1] ;
	}
	\newcommand{\doublelobe}{
	    \lobe
	    \begin{scope}[xshift=1cm]
		\lobe
	    \end{scope}
	}
	\begin{scope}[xshift=3cm, rotate around = {45:(-1,0)}]
	    \begin{scope}
		\draw (-1,3) -- (-1,0) -- (1,0) -- (1,3) ;
		\draw (1,3) arc[start angle=0, end angle=180, radius=1] ;
	    \end{scope}
	    \begin{scope}[xshift=1cm, opacity=0.4]
		\filldraw (-1,3) -- (-1,0) -- (1,0) -- (1,3) ;
		\filldraw (1,3) arc[start angle=0, end angle=180, radius=1] ;
	    \end{scope}
	    \doublelobe
	\end{scope}
	\begin{scope}[rotate around = {-45:(2,0)}]
	    \begin{scope}
		\draw (-1,3) -- (-1,0) -- (1,0) -- (1,3) ;
		\draw (1,3) arc[start angle=0, end angle=180, radius=1] ;
	    \end{scope}
	    \begin{scope}
		\draw (-1,3) -- (-1,0) -- (1,0) -- (1,3) ;
		\draw (1,3) arc[start angle=0, end angle=180, radius=1] ;
	    \end{scope}
	    \doublelobe
	\end{scope}
	\begin{scope}[yshift=0.707cm, rotate around={-45:(2,0)}, transform shape]
	    \foreach \x in {-1, 0,1,2} {
		\foreach \y in {0,1,2,3} {
		    \draw (\x,\y) node(\x\y) { } ; 
		}
	    }
	\end{scope}
    \end{tikzpicture}
    \hspace{1cm}
    \begin{tikzpicture}[scale=0.5]
	\newcommand{\lobe}{
	    \draw (-1,3) -- (-1,0) -- (1,0) -- (1,3) ;
	    \draw (1,3) arc[start angle=0, end angle=180, radius=1] ;
	}
	\newcommand{\doublelobe}{
	    \lobe
	    \begin{scope}[xshift=1cm]
		\lobe
	    \end{scope}
	}
	\begin{scope}[xshift=3cm, rotate around = {45:(-1,0)}]
	    \begin{scope}
		\draw (-1,3) -- (-1,0) -- (1,0) -- (1,3) ;
		\draw (1,3) arc[start angle=0, end angle=180, radius=1] ;
	    \end{scope}
	    \begin{scope}[xshift=1cm, color=blue,opacity=0.45]
		\draw (-1,3) -- (-1,0) -- (1,0) -- (1,3) ;
		\draw (1,3) arc[start angle=0, end angle=180, radius=1] ;
	    \end{scope}
	    \doublelobe
	\end{scope}
	\begin{scope}[rotate around = {-45:(2,0)}]
	    \begin{scope}[opacity=0.4]
		\filldraw (-1,3) -- (-1,0) -- (1,0) -- (1,3) ;
		\filldraw (1,3) arc[start angle=0, end angle=180, radius=1] ;
	    \end{scope}
	    \begin{scope}[xshift=1cm]
		\draw (-1,3) -- (-1,0) -- (1,0) -- (1,3) ;
		\draw (1,3) arc[start angle=0, end angle=180, radius=1] ;
	    \end{scope}
	    \doublelobe
	\end{scope}
	\begin{scope}[yshift=0.707cm, rotate around={-45:(2,0)}, transform shape]
	    \foreach \x in {-1, 0,1,2} {
		\foreach \y in {0,1,2,3} {
		    \draw (\x,\y) node(\x\y) { } ; 
		}
	    }
	\end{scope}
    \end{tikzpicture}
    \hspace{1cm}
    \begin{tikzpicture}[scale=0.5]
	\newcommand{\lobe}{
	    \draw (-1,3) -- (-1,0) -- (1,0) -- (1,3) ;
	    \draw (1,3) arc[start angle=0, end angle=180, radius=1] ;
	}
	\newcommand{\doublelobe}{
	    \lobe
	    \begin{scope}[xshift=1cm]
		\lobe
	    \end{scope}
	}
	\begin{scope}[xshift=3cm, rotate around = {45:(-1,0)}]
	    \begin{scope}
		\draw (-1,3) -- (-1,0) -- (1,0) -- (1,3) ;
		\draw (1,3) arc[start angle=0, end angle=180, radius=1] ;
	    \end{scope}
	    \begin{scope}
		\draw (-1,3) -- (-1,0) -- (1,0) -- (1,3) ;
		\draw (1,3) arc[start angle=0, end angle=180, radius=1] ;
	    \end{scope}
	    \doublelobe
	\end{scope}
	\begin{scope}[rotate around = {-45:(2,0)}]
	    \begin{scope}
		\draw (-1,3) -- (-1,0) -- (1,0) -- (1,3) ;
		\draw (1,3) arc[start angle=0, end angle=180, radius=1] ;
	    \end{scope}
	    \begin{scope}[xshift=1cm,opacity=0.4]
		\filldraw (-1,3) -- (-1,0) -- (1,0) -- (1,3) ;
		\filldraw (1,3) arc[start angle=0, end angle=180, radius=1] ;
	    \end{scope}
	    \doublelobe
	\end{scope}
	\begin{scope}[yshift=0.707cm, rotate around={-45:(2,0)}, transform shape]
	    \foreach \x in {-1, 0,1,2} {
		\foreach \y in {0,1,2,3} {
		    \draw (\x,\y) node(\x\y) { } ; 
		}
	    }
	\end{scope}
    \end{tikzpicture}
\end{center}

In Figure~\ref{fig:hearts} we provide a catalog of diagrams depicting eleven types of unsplittable configurations with four or fewer sets. We use the following abbreviations: the symbol $o$ denotes an odd number of elements, $e$ denotes an even number of elements, $1$ denotes one element, $0/1$ denotes zero or one element, $x$ denotes any number of elements, and a blank denotes zero elements. Note that two separate instances of a symbol do \emph{not} necessarily denote the same quantity.

\begin{figure}[ht]
  \input{hearts.tex}
  \caption{Eleven types of unsplittable four-set configurations.\label{fig:hearts}}
\end{figure}

Some additional remarks about the types are in order. First, each of the types has analogous instances in which the sets $B_1,\ldots,B_4$ are permuted.

Next, in Type~$0$ we additionally require that $a_1+a_2$, $b_1+b_2$, and $c_1+c_2$ are all odd numbers. Type~$0$ represents the case when some subcollection of three sets is unsplittable. The instance depicted in the figure shows just the case when the first three sets are unsplittable. There is some overlap between Type~$0$ and degenerate instances of other types; for example an instance of Type~$5$ with all $e$'s being $0$ is also in Type~$0$.

\begin{prop}
  \label{prop:types-unsplittable}
  If $\B$ is a collection of any of the Types $0$--$10$, then $\B$ is unsplittable.
\end{prop}

\begin{proof}[Sketch of proof]
  If $\B$ is of Type~$0$ then $\B$ is unsplittable by Proposition~\ref{prop:oddoddodd}. If $\B$ is of Type~$4$ or~$5$ then $\B$ is unsplittable by Lemma~\ref{lem:parity}, since in each case every element has even multiplicity but the target sum is odd.
  
  If $\B$ is of any of the remaining types, then the proof boils down to elementary linear algebra. As an example, let us suppose that $\B$ is of Type~$1$. Let $o_i$ denote the number of elements of $\bigcap_{j\neq i}B_j\cap B_i^c$, so that $o_i$ is an odd number. Let $t_i=\frac12(1+\sum_{j\neq i}o_i)$ be the target quantity for $S\cap B_i$. Then finding a splitter $S$ for $\B$ is equivalent to solving the integer system:
\[\begin{bmatrix}
    1 & 1 & 1 &   & 1 \\
    1 & 1 &   & 1 & 1 \\
    1 &   & 1 & 1 & 1 \\
      & 1 & 1 & 1 & 1	
  \end{bmatrix}
  \begin{bmatrix} a_1 \\ a_2 \\ a_3 \\ a_4 \\ b \end{bmatrix}
  =
  \begin{bmatrix} t_1 \\ t_2 \\ t_3 \\ t_4 \end{bmatrix}
\]
  subject to the constraints that $0 \leq a_i \leq o_i$, and $0\leq b \leq 1$. Solving this system for $a_1$ in terms of the right-hand sides and $b$, we obtain the equation:
  \begin{align*}
    a_1 &= \frac13(t_1+t_2+t_3-2t_4-b) \\
        &= \frac13\left(\frac{o_1+o_2+o_3+1}{2} +\frac{o_1+o_3+o_4+1}{2}
           +\frac{o_1+o_2+o_4+1}{2} - 2\frac{o_2+o_3+o_4+1}{2}-b\right)\\
        &= \frac12o_1 + \frac16(1-b)
  \end{align*}
  If $b=0$ then this implies $a_1=\frac16(3o_1+1)$. This is a contradiction since $o_1$ odd implies that $3o_1+1 \not\equiv 0 \pmod 6$. On the other hand if $b=1$ then $a_1=\frac12o_1$, which is impossible because $o_1$ is odd.
\end{proof}

We believe that the types shown in Figure~\ref{fig:hearts} completely capture the unsplittable collections of four or fewer sets. The following statement officially has the status of a conjecture, as our justification relies on the use of a computer program for which we have not verified the correctness of the algorithm, implementation, or runtime in a formal way.

\begin{statement}
  \label{stmt:4sets}
  If $\B$ is a four-set, unsplittable collection, then $\B$ falls into one of the Types $0$--$10$.
\end{statement}

We provide a justification that mixes the Reduction Lemma and an exhaustive search. To begin, first note that our program has tested the splittability of every collection of four sets such that every Venn region has size $\leq3$. The code is available in an online repository; see \cite{code}. The program completed successfully on a computing cluster equivalent to several dozen modern laptops in about one day.

Now suppose that $\mathcal B$ is a collection of four sets which is not of any of the Types~$0$--$10$. We wish to show that $\mathcal B$ is splittable. Let $\mathcal B^{(2)}$ be a  collection obtained from $\mathcal B$ emptying each Venn region $R$ that is even in $\mathcal B$, and leaving just $1$ element in each Venn region $R$ that is odd in $\mathcal B$. In other words, $\mathcal B^{(2)}$ is obtained by taking each Venn region ``modulo~$2$''.

If $\mathcal B^{(2)}$ is also not of any of the eleven types, then since its regions all have size $\leq3$ our program has checked that $\mathcal B^{(2)}$ is splittable. By the Reduction Lemma (Lemma~\ref{lem:reduction}), $\mathcal B$ is also splittable, and we are done in this case.

On the other hand, suppose that $\mathcal B^{(2)}$ is of one of the eleven types, say type $T$. Then since $\mathcal B$ is not of type~$T$, there must exist a Venn region $R$ such that type~$T$ prescribes that $R$ has at most $1$ element, and such that the size of $R$ in $\mathcal B^{(2)}$ is strictly less than the size of $R$ in $\mathcal B$. (In the notation of the figure, region $R$ must be labeled empty, $0/1$, or $1$.) Now let $\mathcal B'$ be the configuration obtained from $\mathcal B^{(2)}$ by adding $2$ elements to $R$. Thus $\mathcal B'$ is not among the Types~$0$--$10$. And since the regions of $\B'$ still have size $\leq3$, our program has checked that $\mathcal B'$ is splittable. It again follows from Lemma~\ref{lem:reduction} that $\mathcal B$ is splittable

This concludes the justification of Statement~\ref{stmt:4sets} modulo the correctness of our computer program.

\subsection{The prevalence of splittability}

In this subsection we address several questions about how commonly splittable and unsplittable collections occur. Our results for small collections of sets indicate that unsplittability is very rare. It is natural to ask whether this remains true for collections with a larger number of sets.

To begin, if one looks at the types of four-set unsplittable collections in Figure~\ref{fig:hearts}, one might surmise that unsplittable configurations should have many Venn regions with few or zero elements. We next establish that if a collection has certain Venn regions with sufficiently many elements, then the collection must in fact be splittable.

\begin{thm}
  \label{thm:dmono}
  Let $D$ be an integer bound on the discrepancy of collections of $n$ sets. Suppose that $\B$ is a collection of $n$ sets such that each Venn region of multiplicity~$1$ contains at least $D-1$ elements. Then $\B$ is splittable.
\end{thm}

\newcommand{\D}{\mathcal{D}}
\begin{proof}
  Let $\B = \{B_1, \ldots, B_n\}$ be such a collection, and let $R_i$ denote the Venn region of multiplicity~$1$ contained in $B_i$. Then in $\B$, each region $R_i$ has at least $D-1$ elements, so we may let $\B^{(0)}=\{B_1^{(0)},\ldots,B_n^{(0)}\}$ be the collection obtained from $\B$ by deleting $D-1$ elements from each of the Venn regions $R_i$.
    
  Since $\disc(\B^{(0)})\leq D$, we can find a set $S^{(0)}$ such that for all $i$ we have\
  \[-D\leq|B_i^{(0)}\cap S^{(0)}|-|B_i^{(0)}\setminus S^{(0)}|\leq D
  \]
  Now for each $i$ we restore the $D-1$ deleted elements of the Venn region $R_i$. As we do so, we build a set $S$ by beginning with $S^{(0)}$, then placing \emph{some} of the $D-1$ restored elements into $S$ and the rest into $S^c$. It is easy to do so in such a way that for each $i$,
  \[-1\leq|B_i\cap S|-|B_i\setminus S|\leq 1
  \]
  and as a result $S$ splits $\B$.
\end{proof}

Of course in the above result, $D$ can be taken to be the ceiling of Spencer's bound $6\sqrt{n}$ discussed in the introduction. 

Now let $f(n,k)$ denote the fraction of all $n$-set collections on $k$ elements which are splittable. The next result implies that if we fix $n$ and let $k$ get large, then $f(n,k)$ converges to $1$. In fact, the same holds even if we let $n,k\to\infty$ with $k$ growing fast enough with respect to $n$.

\begin{thm}
  Suppose that $k=k(n)$ lies in $\omega(2^nn)$, that is, $k$ grows asymptotically strictly faster than $2^nn$. Then $f(n,k)\to1$ as $n\to\infty$.
\end{thm}

\begin{proof}
  Referring to Theorem~\ref{thm:dmono} above, let $D=D(n)=6\sqrt{n}$. Our strategy is to show that if $k$ is as large as in our hypothesis, then it is unlikely that any of the multiplicity~$1$ regions will contain fewer than $D$ elements. Thus by Theorem~\ref{thm:dmono} it is likely that a given configuration will be splittable.

  For this note that if a single element is randomly assigned to be included in or in excluded from each of $n$ sets, then the probability that the element will lie in any given Venn region is $q=\frac{1}{2^n}$. Next assign $k$ elements randomly and independently to the sets, and let the random variable $X$ denote the number elements of a fixed Venn region of multiplicity $1$. By the basic properties of the binomial distribution, the expected value of $X$ is $\mu=kq=\frac{k}{2^n}$ and the standard deviation of $X$ is $\sigma=\sqrt{kq(1-q)}=\sqrt{k\frac{1}{2^n}(1-\frac{1}{2^n})}$.

  We now wish to bound from above $\Pr[X<D]$. Letting $t=\frac{\mu-D}{\sigma}$, we have that $\Pr[X<D]\leq Pr[|X-\mu|\geq t\sigma]$. Chebyshev's inequality now gives
  \[\Pr[X<D] \leq \frac{1}{t^2}
  \]
  Substituting the expressions for $t,$ $\mu$, $\sigma$, $D$, and simplifying we obtain
  \begin{align*}
    \Pr[X<D]
    &\leq\frac{\frac{k}{2^n}(1-\frac{1}{2^n})}{(\frac{k}{2^n}-6\sqrt{n})^2}\\
    &\leq \frac{2^n}{k-12\sqrt{n}}
  \end{align*}
  Our hypothesis about the growth of $k$ implies that the latter quantity is $o(1/n)$. Finally the probability that any of the $n$ regions of multiplicity $1$ has fewer than $D$ elements is bounded by $n\Pr[X<D]$ and is thus $o(1)$, or in other words, converges to $0$.
\end{proof}

We remark that if we instead fix $k\geq 3$ and let $n$ become large, then $f(n,k)$ converges to $0$. This is simply because there exists a configuration with three elements that is unsplittable, namely $\B_0=\{\{1,2\},\{1,3\},\{2,3\}\}$. Thus as the number of sets $n$ increases, it becomes very likely that the restriction of the collection to the points $1,2,3$ will contain $\B_0$, and therefore that the collection will be unsplittable.


\bibliographystyle{alpha}
\bibliography{bib}

\end{document}